\documentclass[preprint,authoryear,12pt]{elsarticle}

\usepackage{amssymb,amsmath}
\usepackage{amsthm}
\usepackage{url}
\usepackage{lineno}

 \theoremstyle{definition}
 \newtheorem{theorem}{Theorem}
 
 \theoremstyle{remark}
 \newtheorem{example}{Example}

 \DeclareMathOperator{\Cov}{Cov}
 \DeclareMathOperator{\Co}{co}
 \DeclareMathOperator{\Derivative}{D}
 \DeclareMathOperator{\E}{E} 
 \DeclareMathOperator{\Image}{im} 
 \DeclareMathOperator{\Kexp}{\exp_{\kappa}}
 
 \DeclareMathOperator{\Phiexp}{\exp_{\phi}}
 \DeclareMathOperator{\Philn}{\ln_{\phi}}
 \DeclareMathOperator{\Span}{Span}
 \newcommand{\convexof}[1]{\Co\left(#1\right)}
 \newcommand{\covat}[3]{\Cov_{#1}\left(#2,#3\right)}
 \newcommand{\derivby}[1]{\frac{d}{d#1}}
 \newcommand{\dirderat}[3]{\Derivative  {#1} \left(#2\right) #3}
 \newcommand{\dirder}[2]{\Derivative  {#1} #2}
 \newcommand{\expectat}[2]{{\E}_{#1}\left[#2\right]}
 \newcommand{\expof}[1]{\exp\left(#1\right)}
 \newcommand{\kexpof}[1]{\Kexp\left(#1\right)}
 
 \newcommand{\phiexpof}[1]{\Phiexp\left(#1\right)}
 \newcommand{\philnof}[1]{\Philn\left(#1\right)}
 \newcommand{\reals}{\mathbb{R}}
 \newcommand{\setof}[2]{\left\{#1 \colon #2 \right\}}
 \newcommand{\set}[1]{\left\{#1\right\}}
 \newcommand{\spanof}[1]{\Span\left(#1\right)}
 \newcommand{\sspace}[1]{\mathcal #1}

\journal{EES}

\begin{document}

\begin{frontmatter}



\title{Marginal Polytope of a Deformed Exponential Family}


\author{Giovanni Pistone}

\address{Collegio Carlo Alberto, Via Real Collegio 30, 10024 Moncalieri, Italy}

\begin{abstract}
A deformed logarithm function called $q$-logarithm has received considerable attention by physicist after its introduction by C. Tsallis. J. Naudts has proposed a generalization called $\phi$-logarithm and he has derived the basic properties of $\phi$-exponential families. In this paper we study the related notion of marginal polytope in the case of a finite state space.  
\end{abstract}


\end{frontmatter}


\section{Introduction}

In Sec. \ref{sec:background} we review the basic properties of deformed exponential families as they are discussed in \citet{naudts:2011GTh}. We do not give here detailed references and refer to the bibliographical notes in this monograph. The presentation of exponential families we use is non-parametric as in \citet{pistone:2009EPJB}. In Sec. \ref{sec:marginal-polytope} we give our generalization of the notion of marginal polytope. The standard version is in \citet{brown:86}. 

\section{Background}
\label{sec:background}

The deformed exponential function are usually introduced starting from the logarithm, see e.g. \citet[Ch. 10]{naudts:2011GTh}, with the aim to define a generalization of the entropy. Let $\phi \colon \reals_> \to \reals_>$ be a positive, increasing, absolutely continuous function. The $\phi$-\emph{logarithm} is the function $\philnof v = \int_1^v {dy}/{\phi(y)}$, $v \in \reals_>$, which reduces to ordinary logarithm when $\phi(v) = v$.

The deformed logarithm $\Philn$ is defined on $\reals_>$ and it is strictly increasing and concave. In the following we assume $-\int_{-\infty}^1 \frac{dy}{\phi(y)}=\int_1^{+\infty} \frac{dy}{\phi(y)} = +\infty$ so that the range of $\Philn$ is the full real line $\reals$.

The $\phi$-\emph{exponential} or \emph{deformed exponential} is the inverse function $\Phiexp = \Philn^{-1}$. It is increasing and convex and solves the differential equation $y'=\phi(y)$, $y(0) = 1$. It is convenient to introduce the rate function
\begin{equation*}
  \psi(u) = \frac{\phi(\phiexpof u)}{\phiexpof u},
\end{equation*}
so that we have
\begin{equation*}
  \Phiexp'(u) = \phi(\phiexpof u) = \psi(u) \phiexpof u.
\end{equation*}

Viceversa, given a positive absolutely continuous rate function $\psi \colon \reals \to \reals_>$, such that $\int_{-\infty}^0 \psi(x) dx = \int_{0}^{+\infty} \psi(x) dx = +\infty$ and $\psi'+\psi^2 \ge 0$, the solution of the differential equation $y'(u)=\psi(u)y(u)$, $y(0)=1$, is a deformed exponential $\Phiexp$ whose $\phi$-function is
\begin{equation*}
  \phi(v)=\psi(\philnof v), \quad v\in\reals_>.
\end{equation*}
This deformed exponential is \emph{self-dual}, $\phiexpof{u}\phiexpof{-u}=1$, if, and only if, the rate function $\psi$ is symmetric because
\begin{multline*}
\derivby u {\phiexpof{u}\phiexpof{-u}} = \\ \psi(u)\phiexpof{u}\phiexpof{-u}-\psi(-u)\phiexpof{u}\phiexpof{-u} = \\ \left(\psi(u)-\psi(-u)\right)\phiexpof{u}\phiexpof{-u} .
\end{multline*}

\begin{example}[Kaniadakis]
The deformed exponential defined in \citet{kaniadakis:2001PhA} with parameter is
\begin{equation*}
  \kexpof u = \expof{\int_0^u \frac{dy}{\sqrt{1+\kappa^2 y^2}}}, \quad u \in \reals, \quad \kappa \in [0,1[.
\end{equation*}
It has
\begin{equation*}
  \psi(y) = (1+\kappa^2 y^2)^{-1/2},\quad
 \phi(x) = \frac{2x}{x^\kappa+x^{-\kappa}}.
\end{equation*}
\end{example}

\begin{example}[Tsallis]
The function $\phi(y)=y^q$ does not satisfy our restrictive assumptions and has been used by Tsallis in 1994 to define his deformed logarithm. Given a symmetric function $\sigma$, then $\phi(y)=y\sigma(y^q,y^{-q})$ is self-dual. Kaniadakis' logarithm is an example, with $\sigma(s,t)=2/(s+t^{-1})$.
\end{example}

Let $(\sspace X,\mu)$ be a finite sample space with reference measure $\mu$. We denote by $L$ the vector space of real random variables on $\sspace X$. If $p$ is any probability density of the sample space, $L_0(p)$ denotes the vector space of $p$-centered random variable, i.e. $u\in L$ belongs to $L_0(p)$ if $\expectat p u = \sum_{x\in\sspace X} u(x) p(x) \mu(x) = 0$. Given a density $p$ and random variables $H_j\in L$, $j=1,\dots,m$, the $\phi$-\emph{exponential family} of the \emph{statistics} $H_j$ is the parametrized set of densities
\begin{equation}
\label{eq:phiexppar}
  p_\theta(x) = \phiexpof{\sum_{j=1}^m \theta_j H_j(x) - \alpha(\theta)}p(x), \quad \theta\in \reals^m,
\end{equation}
where $\alpha(\theta)$ is characterized by the normalization condition. In fact, the function
\begin{equation*}
  \alpha \mapsto \expectat p {\phiexpof{\sum_{j=1}^m \theta_j H_j - \alpha}}
\end{equation*}
is continuous and decreasing from $+\infty$ to 0, so that for each $\theta\in\reals^m$ there exists a unique value $\alpha(\theta)$ such that $p_\theta$ in Eq. \eqref{eq:phiexppar} is a density. 

Two different sets of statistics $H_j$, $j=1,\dots,m$ and $H'_j$, $j=1,\dots,m'$ define the same statistical model if, and only if, the vector space generated by the centered random variables is the same,
\begin{equation*}
  \spanof{H_j - \expectat p {H_j},j=1,\dots,m} = \spanof{H'_j - \expectat p {H'_j}, j=1,\dots,m'}.
\end{equation*}

According to \citet{pistone:2009EPJB}, it is more canonical to define the $\phi$-exponential model as the set of densities $p_u$ of the form
\begin{equation}\label{eq:kexpmodel}
 p_u = \phiexpof{u - K(u)} p, \quad u \in V,
\end{equation}
where $V=\spanof{H_j - \expectat p {H_j},j=1,\dots,m}$ is a linear sub-space of $L_0(p)$ and
\begin{equation}\label{eq:u-theta}
  K(u) = \alpha(\theta) - \sum_{j=1}^m \theta_j\expectat p {H_j}, \quad u = \sum_{j=1}^m \theta_j\left(H_j - \expectat p {H_j}\right) 
\end{equation}

The quantity $K(u)$ is in fact the divergence of $p$ from $p_u$ because from $u - K(u) = \philnof{\frac{q}{p}}$, $\expectat p u = 0$, and the self-duality of the deformed logarithm, we have
\begin{equation*}
  K(u) = \expectat p {\philnof{\frac{p}{q}}}.
\end{equation*}
The random variable $u\in V$ is a unique parameterization of $p_u$ as
\begin{equation*}
u = \philnof{\frac{q}{p}} - \expectat p {\philnof{\frac{q}{p}}}.   
\end{equation*}

The non parametric derivative of the mapping $L_0 \colon u \mapsto K(u)$ is obtained by derivation of the expected value of Eq. \eqref{eq:kexpmodel} in the direction $v\in L_0$. We use the notation $\dirderat  K u v = \left.\derivby t K(u+tv)\right|_{t=0}$. As $\expectat \mu {p_u} = \expectat p {\phiexpof{u - K(u)}} = 1$, we obtain
\begin{multline*}
  \expectat p {\dirder {\phiexpof{u - K(u)}}v} = \\ \expectat p {\psi(u - K(u))\phiexpof{u - K(u)} \dirder{(u-K(u))}{v}} = 0.
\end{multline*}
As we have
\begin{equation*}
  \psi(u - K(u)) = \psi\left(\philnof{\frac{p_u}{p}}\right) = \phi\left(\frac{p_u}{p}\right)
\end{equation*}
we can write
\begin{equation*}
  \expectat {p_u} {\phi\left(\frac{p_u}{p}\right) v} = \expectat {p_u} {\phi\left(\frac{p_u}{p}\right)} \dirderat K u v. 
\end{equation*}

The probability density
\begin{equation*}
  p_{\phi,u} = \frac{\phi\left(\frac{p_u}{p}\right)}{\expectat {p_u} {\phi\left(\frac{p_u}{p}\right)}} p_u
\end{equation*}
is called the \emph{escort density} and we have $\dirderat K u v = \expectat {\phi,u} v$.

The second derivative of $u \mapsto \phiexpof{u - K(u)}$ in the directions $v$ and $w$ is the first derivative in the direction $w$ of $u \mapsto \Phiexp'(u - K(u))(v - DK(u) v)$, therefore it is equal to
\begin{equation}
  \label{eq:30}
\Phiexp''(u - K(u))(v-DK(u)v)(w-DK(u)w) -\Phiexp'(u - K(u)) D^2K(u) vw.
\end{equation}
The random variable in Eq. \eqref{eq:30} has zero $p$-expectation, so that 
\begin{equation*}
  D^2K(u) vw = \frac{\expectat p {\Phiexp''(u - K(u))(v-DK(u)v)(w-DK(u)w)}}{\expectat p {\Phiexp'(u - K(u))}}.
\end{equation*}
If $w=v\ne0$, then $D^2K(u)vv > 0$, therefore the functional $K$ is strictly convex. For $u=0$ we obtain $D^2 K(0) vw = \covat p u v$. We do not have a similar interpretation for $u \ne 0$, but see the discussion of the parallel transport in \citet[sec. 4]{pistone:2009EPJB} and of conformal transformations in \citet{ohara|amari:2011}.

The convex conjugate of $L_0(p) \colon u \mapsto K(u)$, is defined in the duality $(u^*,u) \mapsto \expectat p {u^* u}$ by
\begin{equation}
  \label{eq:conjugateK}
  H(u^*) = \sup\setof{\expectat p {u^* u} - K(u)}{u\in L_0(p)}, \quad u^* \in L_0(p).
\end{equation}

The function $L_0(p) \ni u \mapsto \expectat p {u^* u} - K(u)$ is concave and has derivative in the direction $v$ equal to
\begin{equation*}
  \expectat p {u^* v} - \expectat{\phi,u}{v} =  \expectat p {u^* v} - \expectat{p}{\frac{p_{\phi,u}}p v}.
\end{equation*}
 
If we have a maximum at $\hat u$ in Eq. \eqref{eq:conjugateK}, then $\hat u$ satisfies $p_{\phi,\hat u} = u^* p$. 

For the model in Eq. \eqref{eq:kexpmodel} we define
\begin{equation}
  \label{eq:conjugateKV}
  H_V(u^*) = \sup\setof{\expectat p {u^* u} - K(u)}{u\in V}, \quad u^* \in L_0(p).
\end{equation}
In this case the directional derivative has to be zero for all $v\in V$ and a maximum at $\hat u$ implies that $u^*$ has orthogonal projection on $V$ equal to that of $p_{\phi,\hat u}/p$.

\section{Marginal polytope}
\label{sec:marginal-polytope}

In this section we consider the $\phi$-exponential family of Eq. \eqref{eq:phiexppar} together with its non-parametric presentation, where we have a vector space $V$ of $p$-centered random variables and the statistical model of Eq. \eqref{eq:kexpmodel}. The random variables $H_j-\expectat p {H_j}$, $j=1,\dots,m$, span the space $V\subset L_0(p)$ and the relation beween the two parameterization $\theta\in \reals^m$ and $u\in V$ are shown in \eqref{eq:u-theta}. The \emph{marginal polytope} of the $\phi$-model (also called \emph{convex support}) is the convex hull $M$ of the set $\setof{H(x)}{x \in \sspace X} \subset R^m$, $H=(H_1,\dots,H_m)$. The function $\alpha \colon \reals^m \to \reals$ is convex and has a convex conjugate $\alpha^*(\eta) = \sup\setof{\theta\cdot\eta-\alpha(\theta)}{\theta\in\reals^m}$.

The following theorem is the deformed version of classical results e.g, \cite[Th. 3.6]{brown:86}. We use notions of convex analysis to be found in \citet{rockafellar:1970}.
\begin{theorem}\label{th:M}
\begin{enumerate}
\item \label{item:1} The convex conjugate $\alpha^* \colon \reals^m \to \reals\cup\set{+\infty}$ of $\alpha$ is finite exactly on the marginal polytope $M = \convexof{\Image H}$.
\item \label{tem:2} The gradient mapping $\nabla \alpha \colon \reals^m \to \reals^m$ is onto the interior $M^\circ$ of the marginal polytope $M$.
\item \label{item:3} $\alpha^*$ restricted to $M^\circ$ is the Legendre transform of $\alpha$ that is, $\alpha^*(\eta) = \theta\cdot\eta - \alpha(\theta)$ if $\eta = \expectat {\phi,\theta} H$. 
\end{enumerate}
\end{theorem}
\begin{proof}
\begin{enumerate}
\item Assume first $\eta \in M$, namely $\eta = \sum_{x} \lambda(x) H(x)$, $\lambda(x) \ge 0$, $\sum_x \lambda(x)=1$, and  $H(x)=(H_1(x),\dots,H_m(x))$. From the convexity of the $\Phiexp$ function, we obtain
  \begin{multline*}
\phiexpof{\theta\cdot\eta-\alpha(\theta)} = \phiexpof{\sum_x \lambda(x)\left(\theta\cdot H(x)-\alpha(\theta)\right)} \le \\ \sum_x \lambda(x) \phiexpof{\theta\cdot H(x)-\alpha(\theta)}
  \end{multline*}
As $\mu(x) > 0$, $x \in \sspace X$, $C = \max_x \lambda(x)/\mu(x)$ is finite and $\lambda(x) \le C \mu(x)$, therefore
\begin{equation*}
\phiexpof{\theta\cdot\eta-\alpha(\theta)} \le C \sum_x \mu(x) \phiexpof{\theta\cdot H(x)-\alpha(\theta)} = C,
\end{equation*}
so that $\theta\cdot\eta-\alpha(\theta) \le \philnof C$ for all $\theta$.

Assume now $\eta \notin M$. As $M=\convexof{T(x) \colon x\in\sspace X}$ is a compact convex set, $\eta$ is strictly separated from $M$, that is, there exist an affine function $t \mapsto a \cdot t - a_0$, $a=(a_1,\dots,a_m)$, such that $a \cdot H(x) \le a_0$, $x \in \sspace X$, and $a\cdot\eta = a_0+1$. Along the sequence $\theta_n=na$, $n=1,2,\dots$, we have
\begin{equation*}
  1 = \sum_x \phiexpof{\theta_n\cdot H(x) - \alpha(\theta_n)}\mu(x) \le \phiexpof{na_0 - \alpha(\theta_n)}, 
\end{equation*}
so that
\begin{equation*}
  0 \le na_0 - \alpha(\theta_n) = \theta_n\cdot \eta - n - \alpha(\theta_n),
\end{equation*}
which in turn implies $\theta_n\cdot\eta - \alpha(\theta_n) \to + \infty$ as $n \to \infty$.
\item Each $\eta = \nabla \psi(\theta)$ is of is an expected value with respect to the escort density, $\eta = \expectat {\phi,\theta} H$, therefore $\eta \in M^\circ$ because the escort density is strictly positive. Viceversa, if $\eta \in M^\circ$, let us assume first that the elements of $H$ are linearly independent so that the convex set $M$ is solid. In such a case, there exists a positive constant $\epsilon > 0$ such that $\eta + \epsilon u \in M^\circ$ for all unit vector $u$. Consider a tangent hyperplane $T_u(y)=0$ of the $\epsilon$-ball centered at $\eta$, orthogonal to $u$, $T_u(y) = u\cdot(y-\eta)-\epsilon$. It follows that $T(\eta)= -\epsilon < 0$ and moreover there exists at least one $x_u \in \sspace X$ such that $T(H(x_u)) > 0$, that is $\theta\cdot(H(x_u)-\eta) > \rho\epsilon$ for $\theta=\rho u$, $\rho\in\reals$. We have
  \begin{equation*}
1 = \sum_x \phiexpof{\theta\cdot H(x) - \alpha(\theta)} \mu(x) \ge \phiexpof{\theta\cdot H(x_u) - \alpha(\theta)} \mu(x_u),
  \end{equation*}
hence
\begin{equation*}
  \philnof{\frac1{\mu(x_u)}} \ge \theta\cdot H(x_u) - \alpha(\theta) > \theta\cdot\eta + \rho\epsilon - \alpha(\theta).  
\end{equation*}
therefore, for $\rho \to +\infty$,
\begin{equation*}
  \theta\cdot\eta - \alpha(\theta) < -\rho \epsilon \to -\infty.
\end{equation*}
The maximum of $\theta \mapsto \theta\cdot\eta - \alpha(\theta)$ is reached at some point $\hat \theta$ that satisfies $\eta = \nabla \alpha(\hat\theta)$. 

The general case is obtained by considering a linear independent subset of $H_j$ and expressing the marginal polytope as an affine function of the reduced marginal polytope.
\item From the convexity of $\alpha$ it follows
  \begin{equation*}
    \alpha(\theta) - \alpha(\hat\theta) \ge \nabla \alpha(\hat\theta)\cdot(\theta - \hat\theta) = \eta\cdot(\theta - \hat\theta). 
  \end{equation*}
By rearranging the terms,
\begin{equation*}
  \eta\cdot\hat\theta - \alpha(\hat\theta) \ge \eta\cdot\theta - \alpha(\theta),
\end{equation*}
therefore, $\alpha^*(\eta) = \eta\cdot\hat\theta - \alpha(\hat\eta)$.
\end{enumerate}
\end{proof}

The non parametric version of item \ref{item:1} of Th. \ref{th:M} follows. We are not discussing here the other two items.
\begin{theorem}\label{th:Mnp}
The convex conjugate $H_V$ of $K$ is finite at $u^*$ if, and only if, $(u^*+1)p$ is a density, that is $\expectat p {u^*} = 0$, $u^*+1 \ge 0$.
\end{theorem}
\begin{proof}
From Eq. \eqref{eq:u-theta} we have for $u\in V$
\begin{align*}
  \expectat p {u^*u} - K(u) &= \sum_{j=1}^m \theta_j \expectat p {u^*(H_j-\expectat p {H_j})} - \alpha(\theta) + \sum_{j=1}^m \theta_j \expectat p {H_j} \\ &= \sum_{j=1}^m \theta_j \expectat p {(u^*+1)H_j} - \alpha(\theta) \\ &= \eta\cdot\theta - \alpha(\theta),
\end{align*}
where $\eta_j= \expectat p {(u^*+1)H_j}$, $j=1,\dots,m$. For such a $\eta$ we have $H_V(u^*) = \sup_\theta (\eta\cdot\theta - \alpha(\theta))$ which is finite if $\eta\in M$. As $\eta=\sum_x (u^*(x)+1)H(x)p(x)$, $(u^*+1)p$ must be a density. 
\end{proof}

\section{Discussion}

We have shown in Th. \ref{th:M} that the basic properties of the marginal polytope carry over from the classical case to the deformed case. Moreover, we have provided in Th. \ref{th:Mnp} a non parametric definition of the marginal polytope. This should be relevant in the discussion of the variational properties of the related deformed entropies. The results are restricted to the case of finite state space, while the non parametric language prompts for a generalization in the direction of \citet{pistone|sempi:95}.

A new general non parametric approach has been used in \citet{vigelis|cavalcante:2011} to ease the discussion of the dependence of the chart from the reference density $p$. Each density $q$ is represented as $\phiexpof{u - K_\text{VC}(u) + \Philn p}$ with respect to $\mu$, $u\in L_0(\mu)$. In such a case we have $\Philn q = u - K_\text{VC}(u) + \Philn p$, that is $K_\text{VC}(u) = \expectat \mu {\Philn p - \Philn q}$ instead of $K(u)=\expectat \mu {\philnof{\frac{p}{q}}}$.

\section*{References}

\end{document}